\documentclass[12pt]{amsart}
\usepackage{latexsym, amssymb,amsthm,amsmath,verbatim}

\begin{document}

\newtheorem{theorem}{Theorem}
\newtheorem{proposition}[theorem]{Proposition}
\newtheorem{question}[theorem]{Question}
\newtheorem{definition}[theorem]{Definition}
\newtheorem{corollary}[theorem]{Corollary}
\newtheorem{lemma}[theorem]{Lemma}
\newtheorem{claim}[theorem]{Claim}
\newtheorem{example}[theorem]{Example}

\newcommand{\pr}[1]{\left\langle #1 \right\rangle}
\newcommand{\mH}{\mathcal{H}}
\newcommand{\mK}{\mathcal{K}}
\newcommand{\mR}{\mathcal{R}}
\newcommand{\mG}{\mathcal{G}}
\newcommand{\mA}{\mathcal{A}}
\newcommand{\mV}{\mathcal{V}}
\newcommand{\mU}{\mathcal{V}}
\newcommand{\mP}{\mathcal{P}}
\newcommand{\mB}{\mathcal{B}}
\newcommand{\C}{\mathrm{C}}
\newcommand{\mO}{\mathcal{O}}
\newcommand{\mC}{\mathcal{C}}
\newcommand{\D}{\mathrm{D}}
\renewcommand{\O}{\mathrm{O}}
\newcommand{\K}{\mathrm{K}}
\newcommand{\OD}{\mathrm{OD}}
\newcommand{\Do}{\D_\mathrm{o}}
\newcommand{\sone}{\mathsf{S}_1}
\newcommand{\gone}{\mathsf{G}_1}
\newcommand{\sfin}{\mathsf{S}_\mathrm{fin}}
\newcommand{\gfin}{\mathsf{G}_\mathrm{fin}}
\newcommand{\gn}[1]{\mathsf{G}_\mathrm{#1}}
\newcommand{\Em}{\longrightarrow}
\newcommand{\menos}{{\setminus}}
\newcommand{\w}{{\omega}}
\newcommand{\1}{\textsc{Alice}}
\newcommand{\2}{\textsc{Bob}}
\newcommand{\seq}[1]{{\langle {#1} \rangle}}

\title[A note on rank 2 diagonals]
{A note on rank 2 diagonals}
\author[A. Bella]{Angelo Bella}
\address{Dipartimento di Matematica e Informatica, University of Catania,
Citt\`a Universitaria, Viale A. Doria 6, 95125 Catania, Italy}
\email{bella@dmi.unict.it}

\author[S. Spadaro]{Santi Spadaro}
\address{Dipartimento di Matematica e Informatica, University of Catania,
Citt\`a Universitaria, Viale A. Doria 6, 95125 Catania, Italy}
\email{bella@dmi.unict.it}

\begin{abstract} We solve two questions regarding spaces with a ($G_\delta$)-diagonal of rank 2. One is a question of Basile, Bella and Ridderbos regarding weakly Lindel\"of spaces with a $G_\delta$-diagonal of rank 2 and the other is a question of Arhangel'skii and Bella asking whether every space with a diagonal of rank 2 and cellularity continuum has cardinality at most continuum.
\end{abstract}

\maketitle
\smallskip

\section{Introduction}

A space is said to have a $G_\delta$-diagonal if its diagonal can be written as the intersection of a countable family of open subsets in the square. This notion is of central importance in metrization theory, ever since Sneider's 1945 theorem \cite{Sn} stating that every compact Hausdorff space with a $G_\delta$ diagonal is metrizable. Sneider's result was later improved by Chaber \cite{Ch} who proved that every countably compact space with a $G_\delta$ diagonal is compact and hence metrizable.

Around the same time, Ginsburg and Woods \cite{GW} showed the influence of $G_\delta$ diagonals in the theory of cardinal invariants for topological spaces by proving that every space with a $G_\delta$ diagonal without uncountable closed discrete sets has cardinality at most continuum. Their result led them to conjecture that every ccc space with a $G_\delta$ diagonal must have cardinality at most continuum. Shakhmatov \cite{Sh} and Uspenskii \cite{U} gave a pretty strong disproof to this conjecture by constructing Tychonoff ccc spaces with a $G_\delta$-diagonal of arbitrarily large cardinality. However, in the meanwhile, several strengthenings of the notion of a $G_\delta$-diagonal had been introduced, leading several researchers to test Ginsburg and Woods's conjecture against these stronger diagonal properties. That culminated in Buzyakova's surprising result \cite{Bu} that a ccc space with a regular $G_\delta$-diagonal has cardinality at most continuum. A space has a \emph{regular $G_\delta$-diagonal} if there is a countable family of neighbourhoods of the diagonal in the square such that the diagonal is the intersection of their closures.

Another way of strengthening the property of having a $G_\delta$ diagonal is by considering the notion of rank. Recall that given a family $\mathcal{U}$ of subsets of a topological space and a point $x \in X$, $St(x,\mathcal{U}):=\bigcup \{U \in \mathcal{U}: x \in U \}$. The set $St^n(x,\mathcal{U})$ is defined by induction as follows: $St^1(x, \mathcal{U})=St(x, \mathcal{U})$ and $St^n(x,\mathcal{U})=\bigcup \{U \in \mathcal{U}: U \cap St^{n-1}(x, \mathcal{U}) \neq \emptyset \}$ for every $n>1$. A space is said to have a diagonal of rank $n$ if there is a sequence $\{\mathcal{U}_k: k < \omega \}$ of open covers of $X$ such that $\bigcap \{St^n(x, \mathcal{U}_k): k < \omega\}=\{x\}$, for every $x \in X$. By a well-known characterization, having a diagonal of rank 1 is equivalent to having a $G_\delta$-diagonal. Note also that a space with a $G_\delta$-diagonal of rank $2$ is necessarily $T_2$.

Zenor \cite{Ze} observed that every space with a diagonal of rank 3 also has a regular $G_\delta$-diagonal so by Buzyakova's result, every ccc space with a diagonal of rank 3 has cardinality at most continuum. In \cite{bella}, the first author proved the stronger result that every ccc space with a $G_\delta$-diagonal of rank $2$ has cardinality at most $2^\omega$. The following question is still open though:

\begin{question}
(Arhangel'skii and Bella \cite{AB}) Is every regular $G_\delta$-diagonal always of rank 2?
\end{question}

A positive answer would lead to a far-reaching generalization of Buzyakova's cardinal bound.

Arhangel'skii and the first-named author proved in \cite{AB} that every space with a diagonal of rank 4 and cellularity $\leq \mathfrak{c}$ has cardinality at most continuum, and leave open whether this is also true for spaces with a diagonal of rank 2 or 3. 

\begin{question}
Let $X$ be a space with a diagonal of rank 2 or 3 and cellularity at most $\mathfrak{c}$. Is it true that $|X| \leq \mathfrak{c}$.
\end{question}

From Proposition 4.7 of \cite{BBR} it follows that $|X| \leq c(X)^\omega$ for every space $X$ with a diagonal of rank 3, which in turn that the answer to Arhangel'skii and Bella's question is yes for spaces with a diagonal of rank 3. We show that the answer to their question is no for spaces with a diagonal of rank 2, by constructing a space with a diagonal of rank 2, cellularity $\leq \mathfrak{c}$ and cardinality larger than the continuum. That leads to a complete solution to Arhangel'skii and Bella's question.

Recall that space $X$ is \emph{weakly Lindel\"of} provided that every open cover has a countable subfamily whose union is dense in $X$.  This notion is a common generalisation of the Lindel\"of property and the countable chain condition (ccc). In view of the results by Ginsburg-Woods and Bella mentioned above it is natural to consider the following question:

\begin{question} \label{q} \cite{BBR}  
Let $X$ be a weakly Lindel\"of space with a $G_\delta$-diagonal of rank $2$. Is it true that $|X| \leq 2^\omega$?
\end{question}

The above question was explicitly formulated in \cite{BBR} and two partial positive answers were obtained there under the assumptions that the space is either Baire or has a rank $3$ diagonal. Here we will prove that Question $\ref{q}$ has a positive answer assuming that the space is normal.

All undefined notions can be found in \cite{eng}.

\section{Spaces with a diagonal of rank 2}

Recall that a neighbourhood assignment for a space $X$ is a
function $\phi$ from $X$  to  its topology such that $x\in
\phi(x)$ for every $x\in X$. A set $Y\subseteq X$ is a \emph{kernel} for
$\phi$ if $X=\bigcup\{\phi(y):y\in Y\}$. Following
\cite{vanmill},   we say that a space $X$ is dually $\mathcal P$
if   every neighbourhood assignment  in $X$ has a kernel $Y$
satisfying the property $\mathcal P$.  Of course,   $\mathcal P$
implies  dually $\mathcal P$. A dually ccc space may fail to be
even weakly Lindel\"of.

Here we need  the countable version of a well-known result of
Erd\" os and Rado:
\begin{lemma} \label{erdos}  Let $X$ be a set with
$|X|>2^\omega$. If $[X]^2=\bigcup\{P_n:n<\omega\}$, then there
exist an uncountable set $S\subseteq X$ and an integer $n_0\in
\omega$ such that $[S]^2 \subseteq P_{n_0}$. \end{lemma}

\begin{theorem} \label{1} If $X$ is a dually weakly Lindel\"of
normal space with a $G_\delta$-diagonal of rank $2$, then $|X|\le
2^\omega$. \end{theorem}
\begin{proof}  Let $\{\mathcal U_n:n<\omega\}$ be a sequence of
open covers of $X$ such that $\{x\}=\bigcap\{St^2(x,\mathcal
U_n):n<\omega\}$ for each $x\in X$. Assume by contradiction that
$|X|>2^\omega$ and for any $n<\omega$  put $P_n=\{\{x,y\}\in
[X]^2: St(x,\mathcal U_n)\cap St(y,\mathcal U_n)=\emptyset \}$.
The assumption that  the sequence $\{\mathcal U_n:n<\omega\}$ has
rank $2$ implies that $[X]^2=\bigcup\{P_n:n<\omega\}$. By Lemma
\ref{erdos} there exists an uncountable set $S\subseteq X$ and an
integer $n_0$  such that $[S]^2\subseteq P_{n_0}$. The collection $\{St(x,\mathcal
U_{n_0}):x\in S\}$ consists of pairwise disjoint open sets. From that it follows that, for any $z\in
X$, the set $St(z,\mathcal U_{n_0})$ cannot meet $S$ in two
distinct points, which implies that the set $S$ is closed and discrete.  

We define a neighbourhood assignment $\phi$ for $X$ as follows:
if $x\in S$ let $\phi(x)=St(x,\mathcal U_{n_0})$ and if $x\in
X\setminus S$ let $\phi(x)=X\setminus S$. Since $X$ is dually
weakly Lindel\"of, there exists a weakly Lindel\"of subspace $Y$
such that $X=\bigcup\{\phi(y):y\in Y\}$. By the way $\phi$ is
defined, it follows that $S\subseteq \bigcup\{\phi(y):y\in Y\cap
S\}$ and hence $S\subseteq Y$. As $X$ is normal, we may pick an open
set $V$ such that $S\subseteq V$ and $\overline V\subseteq
\bigcup\{St(x,\mathcal U_{n_0}):x\in S\}$.   The trace on $Y$ of
the open cover $\{St(x,\mathcal U_{n_0}):x\in X\}\cup
\{X \setminus \overline V\}$ witnesses the failure of the weak
Lindel\"of property on $Y$. This is a contradiction and we are
done. \end{proof}

Related results for the classes of dually ccc spaces and for that of cellular-Lindel\"of spaces were proved in \cite{cinesi} and \cite{BS}.

Finally we will construct a space with a diagonal of rank 2, cellularity at most continuum and cardinality larger than the continuum, thus solving Problem 2 from \cite{AB}. Recall that a $\kappa$-Suslin Line $L$ is a continuous linear order (endowed with the order topology) such that $c(L) \leq \kappa < d(L)$. The existence of a $\kappa$-Suslin Line for every $\kappa \geq \omega$ is consistent with ZFC (Jensen proved that it follows from $V=L$).

\begin{theorem}
($V=L$) There is a space $X$ with a diagonal of rank 2 such that $c(X) \leq \mathfrak{c}$ and $|X| \geq \mathfrak{c}^+$.
\end{theorem}

\begin{proof}
Let $T$ be an $\omega_1$-Suslin Line. Let $S$ be the set of all points of $L$ which have countable cofinality. Since $T$ is a continuous linear order the set $S$ is dense in $T$ and hence $d(S)>\aleph_1$. In particular, $|S|>\aleph_1$. Let $\tau$ be the topology on $S$ generated by intervals of the form $(x,y]$, where $x< y \in S$. Note that $c((S, \tau))= \aleph_1$ and that the space $(S,\tau)$ is first-countable and regular. So applying Mike Reed's \emph{Moore Machine} (see, for example \cite{DR}) to $(S,\tau)$ we obtain a Moore space $\mathcal{M}(S)$ such that $|\mathcal{M}(S)|=|S| > \aleph_1$ and $c(\mathcal{M}(S)) \leq \aleph_1$. Recalling that Moore spaces have a diagonal of rank 2 (see Proposition 1.1 of \cite{ArBu}) and that $\mathfrak{c}=\aleph_1$ under $V=L$, we see that $X=\mathcal{M}(S)$ satisfies the statement of the theorem.
\end{proof}

The above theorem also shows that the assumption that the space is Baire is essential in Proposition 4.5 from \cite{BBR}, thus solving a question asked by the authors of \cite{BBR} (see the paragraph after the proof of Lemma 4.6).

\end{document}